\documentclass[11pt]{amsart}

\usepackage{wasysym}
\usepackage{amsmath}
\usepackage{amssymb}
\usepackage{amsthm}
\usepackage{tikz}
\usetikzlibrary{matrix,arrows,backgrounds,shapes.misc,shapes.geometric,patterns,calc,positioning}
\usetikzlibrary{calc,shapes}
\usepackage{wrapfig}
\usepackage{epsfig}  
\usepackage{color}	 %
\input{xy}
\xyoption{poly}
\xyoption{2cell}
\xyoption{all}

\newcommand{\cale}{\mathcal{E}}

\newcommand{\EC}{\textup{Ext}^2_C(DC,C)}
\newcommand{\Ctilde}{\widetilde{C}}
\newcommand{\hh}{\textup{H}}
\renewcommand{\H}{\textup{H}}
\newcommand{\HH}{\textup{HH}}
\newcommand{\Der}{\textup{Der}}
\newcommand{\Inn}{\textup{Inn}}

\newcommand{\End}{\textup{End}}

\newcommand{\p}{\underline{P}}
\newcommand{\h}{\underline{\HH^1}}

\newtheorem{thm}{Theorem}[section]

\newtheorem{lemma}[thm]{Lemma}
\newtheorem{cor}[thm]{Corollary}

\newtheorem{thmIntro}{Theorem}

\theoremstyle{definition}

\newtheorem{example}{Example}[section]

\theoremstyle{remark}

\newtheorem*{remark*}{Remark}

\numberwithin{equation}{section}

\newcommand{\za}{\alpha}
\newcommand{\zb}{\beta}
\newcommand{\zd}{\delta}

\newcommand{\ze}{\epsilon}
\newcommand{\zg}{\gamma}

\newcommand{\Hom}{\textup{Hom}}
\newcommand{\Ext}{\textup{Ext}}

\begin{document}
\title[Hochschild cohomology of partial relation extensions]{Hochschild cohomology of partial relation extensions}
\author{Ibrahim Assem}
\address{I. Assem, D\'epartement de math\'ematiques, Universit\'{e} de Sherbrooke, Sherbrooke, Qu\'{e}bec, Canada, JIK2R1. }
\email{Ibrahim.Assem@usherbrooke.ca}
\author{Maria Andrea Gatica}
\address{ M. A. Gatica, Departamento de Matem\'atica, Universidad Nacional del Sur, Avenida Alem 1253, (8000) Bah\'{\i}a Blanca, Buenos Aires, Argentina.}
\email{mariaandrea.gatica@gmail.com}
\author{Ralf Schiffler}\thanks{The first author gratefully acknowledges partial support from the NSERC of Canada, the second author is grateful to the PGI of the Universidad Nacional del Sur in Argentina for their support, and the third author gratefully acknowledges support by the NSF-CAREER grant  DMS-1254567, and by the University of Connecticut.}
\address{Department of Mathematics, University of Connecticut, 
Storrs, CT 06269-3009, USA}
\email{schiffler@math.uconn.edu}

\begin{abstract} 
We show how to compute the low Hochschild cohomology groups of a partial relation extension algebra.
\end{abstract}

 \maketitle
%


\section{Introduction}
Cluster-tilted algebras appeared as a gift from the theory of cluster algebras to representation theory. These are finite-dimensional algebras which are endomorphism algebras of tilting objects in the cluster category \cite{BMR}. This class of algebras was much investigated, see for example \cite{ABS,ABS2,AsScSe,AsScSe2,  BMR2,CCS,SS,SS2}, and \cite{AR,ARS,AGST} for results on their Hochschild cohomology. Among the main results is that every cluster-tilted algebra can be written as  trivial extension of a tilted algebra by a bimodule called the relation bimodule \cite{ABS}. This explains why many features of tilted algebras are retained by cluster-tilted algebras. In particular, complete slices of tilted algebras embed as what is called local slices in cluster-tilted algebras \cite{ABS2}. However, unlike  tilted algebras, cluster-tilted algebras are not characterized by the existence of local slices. In an effort to find a larger class of algebras having local slices, the authors of \cite{ABDLS} introduced what are called partial relation extensions which, because of the existence of local slices, share many properties with cluster-tilted algebras.

This paper is devoted to the study  of the low Hochschild cohomology groups of partial relation extensions. We now state our main theorem. 
 Let $C$ be a triangular algebra of global dimension at most 2, and assume that the relation bimodule $E=\EC$ splits as a direct sum of two $C$-$C$-bimodules $E=E'\oplus E''$. Then the trivial extension $B=C\ltimes E'$ is called a \emph{partial} relation extension, while $\Ctilde=C\ltimes E$ is called \emph{the} relation extension of $C$. Further, given an algebra $A$ and an $A$-$A$-bimodule $M$, we denote by $\hh^i(A,M)$ the $i$-th Hochschild cohomology group of $A$ with coefficients in $M$ and we set $\hh^i(A,A)=\HH^i(A)$. Finally, we denote by 
  $\cale(M,A)$ the set of all $A$-$A$-bimodule morphisms $f\colon M\to A$ such that 
 $xf(y)+f(x)y=0$, for all $x,y\in M$. With this notation our main theorem reads as follows.
\begin{thmIntro}\label{thm main}
There exist short exact sequences of $k$-vector spaces
\[
\xymatrix@R10pt{0\ar[r]& \H^0(B,E')\ar[r]&\HH^0(B)\ar[r]^{\varphi^0}&\HH^0(C) \ar[r] &0\\
0\ar[r]&\H^1(B,E')\ar[r]&\HH^1(B)\ar[r]^{\varphi^1}&\HH^1(C) \ar[r] &0\\
0\ar[r]& \H^0(\Ctilde,E'')\ar[r]&\HH^0(\Ctilde)\ar[r]^{\varphi^0}&\HH^0(B) \ar[r] &0\\
0\ar[r]&\H^1(\Ctilde,E'')\oplus\cale(E'',B)\ar[r]&\HH^1(\Ctilde)\ar[r]^{\varphi^1}&\HH^1(B) \ar[r] &0.}
\]
\end{thmIntro}

We recall that, if $C$ is tilted, then $\Ctilde$ is cluster-tilted and the partial relation extension $B$ is then a quotient (and a subalgebra) of a cluster-tilted algebra.

The techniques we use are those of \cite{ARS} and \cite{AGST}. In fact several of our proofs follow directly from results of \cite{AGST}, which suggests that the latter hold in greater generality than originally considered.

As a consequence of this, we give another realization of the group $\HH^1(B) $ as the amalgamated sum of  two morphisms. 

The paper is organized as follows. After a preliminary section \ref{sect 1}, we prove our main theorem in section \ref{sect 2}. Section \ref{sect 3} is devoted to corollaries and examples.

\section{Preliminaries}\label{sect 1}
Throughout this paper, $k$ denotes an algebraically closed field,  all algebras are finite-dimensional over $k$ and have an identity. Given an algebra $C$, we denote by $C^e=C\otimes_k C^{op}$ its enveloping algebra. 
If $Q$ is a quiver, we denote by $kQ$ its path algebra. For a point  $i$ of $Q$, let $e_i$ be the primitive idempotent of $kQ$ corresponding to the stationary path at $i$. We refer the reader to \cite{ASS,S} for general notions and results of representation theory.

\subsection{Hochschild cohomology}
Let $C$ be an algebra and $E$ a $C$-$C$-bimodule which is finite-dimensional over $k$.  The {\em Hochschild complex} is the complex 

{\small \[ 0\rightarrow E \xrightarrow{b^1} \Hom_k(C,E)  \xrightarrow{b^2} \cdots  \rightarrow \Hom_k(C^{\otimes i},E)  \xrightarrow{b^{i+1}} \Hom_k(C^{\otimes {(i+1)}},E) \rightarrow \cdots \]}

\noindent where, for each $i >0, \ C^{\otimes i}$ denotes the $i$-fold tensor product of $C$ with itself over $k$. The map  $ b^1\colon E \rightarrow \Hom_k(C,E)$ is defined by $(b^1x)(c)=cx-xc$ for $x \in E, \ c \in C, $ and $b^{i+1}$ is defined by 
{\small\begin{align*}
  (b^{i+1}f)(c_0 \otimes \cdots \otimes c_{i})&= c_0f(c_1 \otimes \cdots \otimes c_{i}) + \sum_{j=1}^i (-1)^j f(c_0 \otimes \cdots \otimes c_{j-1} c_{j}\otimes \cdots  \otimes c_{i}) \\
&+  (-1)^{i+1}f(c_0 \otimes \cdots \otimes c_{i-1})c_{i}
\end{align*}}

\noindent for a $k$-linear map $f: C^{\otimes i} \rightarrow E$ and elements $c_0, \cdots, c_{i} $ in $C$.

The $i$-{th} cohomology group of this complex is called the \emph{$i$-{th} Hochschild cohomology group} of $C$ with coefficients in $E$, and is denoted by $\hh^i(C,E)$. If $_CE_C={}_CC_C$, then we write $\HH^i(C)= \hh^i(C,C)$.

The first Hochschild cohomology group has the following concrete description.   Let $\Der(C,E)$ be the vector space of all {\em derivations}, that is, $k$-linear maps $d: C\rightarrow E$ such that, for $c,c' \in C$, we have 
\[d(cc')=cd(c') +d(c)c'.\]
 A derivation $d$ is \emph{inner} if there exists $x \in E$ such that $d=[x,-]$. Letting $\Inn(C,E)$ denote the subspace of all inner derivations, we have $\hh^1(C,E) \cong \Der(C,E) / \Inn (C,E)$.

 A derivation $d: C \rightarrow E$ is called \emph{normalized} if, for any primitive orthogonal idempotent $e_i$ in a complete set $\{e_1,\ldots,e_n\}$, we have $d(e_i)=0$ for all $i$. Let $\Der_0(C,E)$ be the subspace of $\Der(C,E)$ of the normalized derivations, and $ \Inn_0(C,E)= \Der_0(C,E) \cap \Inn(C,E)$. Then we also have $\hh^1(C,E) \cong \Der_0(C,E)/ \Inn_0(C,E)$.

\subsection{The Hochschild projection maps}
Let $C$ be a finite-dimensional algebra and  $E$  a finitely generated $C$-$C$-bimodule equipped with  an associative $C$-$C$-bimodule morphism $E\otimes _C E\rightarrow E,\,e\otimes e'\mapsto ee'$. The \emph{split extension}  of $C$ by $E$ is the $k$-algebra $B$ which has the additive structure of $C\oplus E$   and whose product is defined by 
\[ (c,e)(c',e')=(cc',ce'+ec'+ee'). \] 
 If $E^2=0$, then $B$ is the \emph{trivial extension} of $C$ by $E$, which we denote by  $B=C\ltimes E.$
 
 In the special case where $C$ is a triangular algebra of global dimension at most 2, and $E$ is the relation bimodule $E=\EC$, the trivial extension $C\ltimes E$ is called  the {\em relation extension of $C$}.  If $E$
 splits as a direct sum of two $C$-$C$-bimodules $E=E'\oplus E''$, then the trivial extension $B=C\ltimes E'$ is called a \emph{partial relation extension of $C$}.

Given a split extension $B$ of $C$ by $E$, there is an exact sequence of vector spaces
\[ \xymatrix{0\ar[r]&E\ar[r]^{i} & B\ar@/^2pt/[r]^p&C\ar@/^2pt/[l]^q\ar[r]&0,} \]
where $p\colon(c,x)\mapsto c$ and $i\colon e\mapsto (0,e)$. Then $p$ is an algebra morphism which has a section $q:c\mapsto (c,0)$. 

Given a $k$-linear morphism $f\colon B^{\otimes n}\to B$,
we have a $k$-linear morphism
$pfq^{\otimes n}\colon C^{\otimes n} \to C$. It is shown in \cite[Corollary 2.2]{AGST} that the assignment $[f]\mapsto[pfq^{\otimes n}]$ defines a $k$-linear map $\varphi^n\colon\HH^n(B)\to\HH^n(C)$, called the $n$-th {\em Hochschild projection morphism.}

\section{Main result}\label{sect 2}
This section is devoted to the proof of Theorem \ref{thm main}. \subsection{}
We start with a criterion for the surjectivity of the first Hochschild projection morphism.

\begin{lemma}
 \label{lem 2.3}
 Let $B$ be a trivial extension of $C$ by $E$. The Hochschild projection morphism $\varphi^1\colon\HH^1(B)\to\HH^1(C)$ is surjective if and only if, for each derivation $d$ of $C$, there exists a $k$-linear map $\za\colon E\to E$ such that
 \[
\begin{array}
 {rclcr} x\,d(c)&=& \za(x)\, c-\za(xc), &&\textup{(C1)}\\
d(c)\,x&=& c\,\za(x)-\za(cx),  &\qquad&\textup{(C2)}
\end{array}
 \]
 for $x\in E,\ c\in C$.
\end{lemma}
\begin{proof}
 This is a reformulation of \cite[Corollary 3.6 (b)]{AGST} in case $n=1$ taking into account that the third condition in loc.cit.~is void in this case. 
\end{proof}

\subsection{}
From now on, we assume that $C$ is a triangular algebra of global dimension two.
\begin{lemma}
 \label{lem 2.4}
 Let $E=E'\oplus E''$ be a decomposition of the $C$-$C$-bimodule $E=\EC$ and $B=C\ltimes E'$ be a partial relation extension. Then 
 $\varphi^1\colon \HH^1(B)\to \HH^1(C)$ is surjective.
\end{lemma}
\begin{proof}
Let $\za\colon C\to C$ be a derivation and $p',q'$ be respectively the canonical projection and inclusion between $E$ and $E'$. Applying Lemma \ref{lem 2.3} above to $\Ctilde=C\ltimes E$, there exists a $k$-linear map $\za\colon E\to E$ which satisfies conditions (C1) and (C2) of the lemma. Let $\za'=p'\,\za\, q'\colon E'\to E'$. This is a $k$-linear morphism. Then for all $x\in E', \ c\in C$, we have
\[\za'(x)\,c-\za'(xc)= (p'\,\za'\,q')(x)\,c- (p'\,\za'\,q')(xc).\]
Considering $x$ and $xc$ as elements of $E$, this expression can be written as
\[p'\za(x)\,c-p'\za(xc) = p'[\za(x)\,c-\za(xc) ],\]
because $p'$ is a morphism of $C$-$C$-bimodules. Now, because of Lemma \ref{lem 2.3}, we have $\za(x)\,c-\za(xc)=x\,d(c)$ inside $E$. Since $x\in E'$ and $d(c)\in C$, we have $x\,d(c)\in E'$. Hence $p'(x\,d(c))=x\,d(c)$. This shows that $\za'(x)\,c-\za'(xc)=x\,d(c)$ as required. The second relation (C2) is proven in the same way.
\end{proof}
\subsection{} We prove the exactness of the first sequence of our main theorem. Here and in the sequel, we keep the notation of Lemma \ref{lem 2.4}, that is, we have a direct sum decomposition $E=E'\oplus E''$ and $B=C\ltimes E'$.
\begin{lemma}
 \label{lem 2.5} There exists a short exact sequence of vector spaces
 \[
\xymatrix@R10pt{0\ar[r]& \H^0(B,E')\ar[r]&\HH^0(B)\ar[r]&\HH^0(C) \ar[r] &0.
}\]
\end{lemma}
\begin{proof}
 Because $C$ is triangular, its center $Z(C)$ is equal to $k$, and hence the bimodule $E'$ is (trivially) symmetric over $Z(C)$, that is, for every $e'\in E'$ and $z\in Z(C)$ we have $z\,e'=e'\,z$. On the other hand, $\HH^0(B)=Z(B)$, $\HH^0(C)=Z(C)$ and $\varphi^0$ is the restriction to $Z(B) $ of the projection $p\colon B\to C$. Thus $\varphi^0$ maps the identity of $B$ to the identity of $C$, hence it is a nonzero morphism. Because $Z(C)=k$, it is surjective, and its kernel is the subspace of $E'$ consisting of all elements which are central in $B$. Thus $\textup{Ker} \,\varphi^0=E'\cap Z(B)$.
 
  We claim that $E' \cap Z(B)\cong\Hom_{B^e}(B,E')$. Indeed, if $f\in \Hom_{B^e}(B,E')$ then $f(1)\in E'\cap Z(B)$ because $f$ is a morphism of $B$-$B$-bimodules. On the other hand, if $x\in E'\cap Z(B)$, then the map $f_x\colon B\to E'$ defined by $1\mapsto x$ is a morphism of $B$-$B$-bimodules, because $x$ is central. It is easily seen that these two maps are inverses to each other. This establishes the claim which implies that $\textup{Ker}\,\varphi^0\cong\Hom_{B^e}(B,E)=\hh^0(B,E')$ as desired.
\end{proof}

\subsection{} The following statement is necessary for the proof of Lemma \ref{lem 2.7}.
\begin{lemma}
 \label{lem 2.6} $\cale(E',C)=0$. 
\end{lemma}
\begin{proof}
 Let $f\in \cale(E',C)$ and define $\overline{f}\colon B\to C$ by $\overline{f}(c,x)=f(x)$, for $(c,x)\in B=C\oplus E'$. Clearly, $\overline{f}|_C=0$. We claim that $\overline{f}$ is a derivation. Let $(c,x),\ (c',x') \in B$. Then
 \[ 
\begin{array}
 {rcl}
 (c,x)\overline{f}(c',x')+\overline{f}(c,x)(c',x') &=& 
 (c,x) f(x')+f(x)(c',x')\\
 &=& (c f(x')+f(x)c', \ xf(x')+f(x)x')\\
 &=& (c f(x')+f(x)c', 0),
\end{array}
 \]
 because $f\in \cale(E',C)$.
 On the other hand, $f$ is a morphism of $C$-$C$-bimodules, hence
 \[\begin{array}
 {rcl}
 (c,x)\overline{f}(c',x')+\overline{f}(c,x)(c',x') &=& 
 (f(c x'+xc'), 0) \\ &=& \overline{f}(cc',cx'+xc') = \overline{f}((c,x)(c',x')). 
\end{array}
 \]
 This completes the proof that $\overline{f}$ is a derivation. Now let $\zg\colon x\to y$ be an arrow in the quiver of $B$ which does not belong to the quiver of $C$. Then $\zg$ is a generator of $E'$ as a $C$-$C$-bimodule. We have
 \[\overline{f}(\zg)=\overline{f}(e_x\zg e_y)=e_x\overline{f}(\zg)e_y\]
 because $e_x, e_y\in C$ imply $\overline{f}(e_x)=0$ and $\overline{f}(e_y)=0$. This shows that $\overline{f}$ maps $e_x E e_y$ to $e_x C e_y$. Moreover the existence of a new arrow $\zg\colon x\to y$ implies that there exists a path from $y$ to $ x$ inside $C$ (in fact a relation), see \cite[Corollary 2.2.1]{ABDLS}. But $C$ is triangular, hence $e_xCe_y=0$. This shows that $\overline{f}=0 $ and hence $f=0$.
\end{proof}
\subsection{} We are now able to prove the exactness of the second sequence of our main theorem.
\begin{lemma}\label{lem 2.7}
 There exists a short exact sequence of vector spaces 
 \[\xymatrix{0\ar[r]&\H^1(B,E')\ar[r]&\HH^1(B)\ar[r]^{\varphi^1}&\HH^1(C) \ar[r] &0.}\]
\end{lemma}
\begin{proof}
 Because of Lemma \ref{lem 2.4}, the projection morphism $\varphi^1$ is surjective. Moreover, we have seen in the proof of Lemma \ref{lem 2.5} that $E'$ is symmetric over $Z(C)$. We apply \cite[Theorem 4.4]{AGST} taking into account that $\cale(E',C)=0$, by Lemma \ref{lem 2.6}.
\end{proof}

\begin{remark*}
 Recall that, by \cite[Proposition 4.8]{AGST}, we have 
\[\hh^1(B,E')=\hh^1(C,E')\oplus \End_{C^e}E.\]
\end{remark*}

\subsection{}
We now continue the proof of our main theorem. Recall that $E=E'\oplus E''$ and $\Ctilde=C\ltimes E$.
\begin{lemma}
 \label{lem 2.9}
 The morphism $\varphi^1\colon\HH^1(\Ctilde)\to \HH^1(B)$ is surjective.
\end{lemma}
\begin{proof}
 Because of \cite[Lemma 2.1.1]{ABDLS}, the morphism $\varphi^1$ is well-defined. Let $d\colon C\to C$ be a derivation and $p'',\ q''$ be respectively the canonical projection and inclusion morphisms between $E$ and $E''$. Because of Lemma \ref{lem 2.3}, there exists $\za\colon E\to E$ which satisfies conditions (C1) and (C2). Consider the $k$-linear map $\za''=p''\,\za\,q''\colon E''\to E''$. Let $x''\in E''$ and $c\in C$.
 \[
\begin{array}
 {rcl}
 \za''(x'')\,c -\za''(x'' c)&=& p''\,\za\,q''(x'')\,c- p''\,\za\,q''(x''c)\\
 &=& p''\,\za(0,x'')\,c- p''\,\za(0,x''c)\\
 &=& p''[\za(0,x'')\,c- \za(0,x''c)],
\end{array}
 \]
 because $p''$ is a morphism of bimodules. Now, in $E''$ we have $\za(0,x'')\,c -\za(0,x''c)=(0,x'') d(c)$ because $\za$ satisfies condition (C1).
 On the other hand, $p''[(0,x'') d(c)]=p''(0,x'')\, d(c)=x''\,d(c)
 $. We have thus proved that 
 \[ \za''(x'')\,c -\za''(x'' c)\ =\ x''\,d(c).\]
 Hence $\za''$ satisfies condition (C1). The proof of condition (C2) is similar.
\end{proof}

\subsection{} The next lemma is needed for the proof of  exactness of the third and the fourth sequence of our main theorem.
\begin{lemma}
 \label{lem 2.10}
 $E''Z(B)=Z(B)E''=0.$ In particular, the bimodule $E''$ is symmetric over $Z(B)$.
\end{lemma}
\begin{proof}
 We must prove that, for each $z\in Z(B)$ and each $e\in E''$, we have $ze=ez=0$. Now $z\in Z(B)$ is a linear combination of nonzero cycles in $B$. Each one of these cycles contains exactly one  arrow in the quiver of $B$ which is not in the quiver of $C$ (for, if it contains none, then it lies in $C$ which is triangular, and if it contains more than one, then it lies in $E^2=0$.)
 On the other hand, $e\in E''$ is a linear combination of nonzero paths, each containing at least one generator of $E''$, that is, an arrow of $\Ctilde$ which is not in $B$. Therefore each path appearing in $ze$ or $ez$ contains at least two arrows in $\Ctilde$ which are not in $C$. Then $ze=ez=0$ because $E^2=0$.
\end{proof}

\subsection{Proof of the main theorem}
We have proved the existence of the first two exact sequences in lemmata \ref{lem 2.5} and \ref{lem 2.7}. The sequence 
\[\xymatrix{
0\ar[r]& \H^0(\Ctilde,E'')\ar[r]&\HH^0(\Ctilde)\ar[r]^{\varphi^0}&\HH^0(B) \ar[r] &0}\]
is exact because of Lemma \ref{lem 2.10} and \cite[Lemma 4.1]{AGST} while the exactness of the sequence
\[\xymatrix{
0\ar[r]&\H^1(\Ctilde,E'')\oplus\cale(E'',B)\ar[r]&\HH^1(\Ctilde)\ar[r]^{\varphi^1}&\HH^1(B) \ar[r] &0}
\]
follows from lemmata \ref{lem 2.9} and \ref{lem 2.10} and \cite[Theorem 4.4]{AGST}.
\qed

\begin{remark*}
(a) In contrast to the second sequence, the term $\cale(E'',B)$ in the fourth sequence does not usually vanish. We refer to Example \ref{ex 2} below.

(b) Moreover, because of \cite[Proposition 4.8]{AGST}, we have
\[\hh^1(\Ctilde,E'')\cong\hh^1(B,E'')\oplus\End_{B^e} E''.\]
\end{remark*}

\section{Corollaries and examples}\label{sect 3}
In this last section, we deduce some consequences of our main result and give a couple of examples. 
\subsection{}  In our first corollary, we give another description of the group $\HH^1(B)$. Because we shall deal at the same time with several bimodule projections and inclusions, we introduce the following notation. Let $X,Y$ be bimodules overt the same algebra, then if there exists a natural inclusion from $X$ to $Y$, it will be denoted by $q_{YX}\colon X\to Y$ and similarly, if there exists a natural projection from $Y$ to $X$, say, it will be denoted by $p_{XY}\colon Y\to X$.

We define a natural morphism $\eta\colon\hh^1(\Ctilde, E) \to \hh^1(B, E')$. Recall that $\hh^1(\Ctilde, E)=\Ext^1_{C^e}(\Ctilde,E)$ while  $\hh^1(B, E')=\Ext^1_{C^e}(B,E')$. Let the exact sequence
\[\xymatrix{ \underline{e}&0\ar[r]&E\ar[r]&X\ar[r]&\Ctilde\ar[r]&0}\] 
represent an element of  $\Ext^1_{C^e}(\Ctilde,E)$, and consider the inclusion morphism $q_{\Ctilde B}\colon B\to \Ctilde$ and the projection $p_{E'\!E}\colon E\to E'$. Then the exact sequence 
\[\xymatrix{ p_{E'\!E}\ \underline{e}\ q_{\Ctilde B}&0\ar[r]&E'\ar[r]&Y\ar[r]&B\ar[r]&0}\] 
represents an element of $\Ext^1_{C^e}(B,E'). $ We set 
$\eta\colon[\underline{e}] \mapsto [ p_{E'\!E}\ \underline{e}\ q_{\Ctilde B}].$

Let also $\psi\colon \HH^1(\Ctilde, E)\to \HH^1(\Ctilde)$ denote the kernel of the Hochschild projection morphism $\varphi^1\colon\HH^1(\Ctilde)\to\HH^1(C)$.

\begin{cor}
 \label{cor 2.12}
 \begin{itemize}
\item [\textup{(a)}] The morphism $\eta $ is surjective and 
\[\textup{Ker}\,\eta=\hh^1(\Ctilde,E'')\oplus \cale(E'',B).\]
\item [\textup{(b)}] $\HH^1(B)$ is the  amalgamated sum of the morphisms $\psi\colon \HH^1(\Ctilde, E)\to \HH^1(\Ctilde)$ and $\eta\colon \HH^1(\Ctilde, E)\to \hh^1(B,E')$.
\end{itemize}
\end{cor}
\begin{proof}
 Because $p_{C\Ctilde}=p_{CB}p_{B\Ctilde}$ and $q_{\Ctilde C}=p_{\Ctilde B}p_{BC}$, the right square of the diagram below commutes.
 \[\xymatrix{0\ar[r]&\H^1(\Ctilde,E')\ar[d]^\eta\ar[r]^\psi&\HH^1(\Ctilde)\ar[d]^{\varphi^1}\ar[r]^{\varphi^1}&\HH^1(C)\ar@{=}[d] \ar[r] &0\\
0\ar[r]&\H^1(B,E')\ar[r]^\zeta&\HH^1(B)\ar[r]^{\varphi^1}&\HH^1(C) \ar[r] &0\\
 }
 \] 
 where all $\varphi^1$ are Hochschild projection morphisms, and the two rows are exact because of \cite[Theorem B]{AGST} and our main theorem above. We  claim that the left square also commutes.
 
 Let $[\underline{e}]\in  \Ext^1_{C^e}(\Ctilde,E)=\hh^1(\Ctilde, E).$ Then, by definition we have $\eta([\underline{e}])=  [ p_{E'\!E}\ \underline{e}\ q_{\Ctilde B}]$. The morphism $\zeta$ is induced from the long exact cohomology sequence, hence 
 $\zeta\eta([\underline{e}])= \zeta( [ p_{E'\!E}\ \underline{e}\ q_{\Ctilde B}]) = [q_{BE'}\ p_{E'\!E} \ \underline{e}\ q_{\Ctilde B}]$.
 On the other side of the square, we have similarly 
 $\varphi^1\psi([\underline{e}])= [p_{B\Ctilde}\ q_{\Ctilde E} \ \underline{e}\ q_{\Ctilde B}]$. It thus suffices to show that $p_{B\Ctilde}\ q_{\Ctilde E}=q_{BE'}\ p_{E'\!E} $, and this follows from the fact that the image of $p_{B\Ctilde}\ q_{\Ctilde E}\colon E\to \Ctilde\to B$ is $E'$. This establishes our claim.
 
 Applying the snake lemma and the surjectivity of $\varphi^1\colon\HH^1(\Ctilde)\to\HH^1(B)$, see our main theorem, we get that $\eta $ is surjective and $\textup{Ker}\,\eta\cong\textup{Ker}\,\varphi^1=\hh^1(\Ctilde,E'')\oplus \cale(E'', B)$. This proves (a), while (b) follows at once from the commutative diagram with exact rows.
\end{proof}

\subsection{}
In the next corollary, we need the algebra structure of $\HH^*(C)=\oplus_{n\ge0}\HH^n(C)$.
 Let $\zeta\in\HH^s(C)$  and $\xi\in \HH^{t}(C)$ be represented by cocycles $f\in\Hom_k(C^{\otimes s},C)$ and  $g\in\Hom_k(C^{\otimes t},C)$, then the {\em cup product}  $\zeta\smile \xi$ is the cohomology class of  the map $f\times g\in \Hom_k(C^{\otimes (s+t)},C)$ defined by \[(f\times g)(c_1\otimes \cdots \otimes c_{s+t})=f(c_1\otimes\cdots \otimes c_s)g(c_{s+1}\otimes \cdots\otimes c_{s+t}).\]
  With this product, $\HH^*(C)$ becomes a graded commutative 
  and associative ring called the \emph{Hochschild cohomology algebra}. It is shown in \cite[Theorem 1]{AGST} that if $B$ is a split extension of $C$, then the Hochschild projection morphisms $\varphi^n$ induce
an algebra morphism  $\varphi^*\colon\HH^*(B)\to\HH^*(C)$.

\begin{cor}
 \label{cor 2.8}
 Let $C$ be a tilted algebra and $B=C\ltimes E'$ a partial relation extension. Then the algebra morphism $\varphi^*\colon\HH^*(B)\to\HH^*(C)$ is surjective and there exists an exact sequence
 \[\xymatrix{0\ar[r]&K
 \ar[r]&  \HH^*(B)\ar[r]^{\varphi^*}&\HH^*(C)\ar[r]&0,}\]
 where $K=\hh^0(B,E')\oplus\hh^1(B,E')\oplus(\oplus_{n\ge 2}\,\HH^n(B)) $.
\end{cor}

\begin{proof}
 Because $C$ is tilted, we have $\HH^*(C)=0$ for all $n\ge 2$, see \cite{Happel}. We apply the exact sequences of Lemmata \ref{lem 2.5} and \ref{lem 2.7}.
\end{proof}

 \subsection{} We recall that in \cite[Remark 2.1.2]{ABDLS}
was defined a poset $\p$ of partial relation extensions. Let $C$ be a triangular algebra of global dimension two, and $\Ctilde=C\ltimes E$ its relation extension. Then a partial relation extension $B_1=C\ltimes E_1$ is said to be smaller than $B_2=C\ltimes E_2$ if $E_1 $ is a direct summand of $E_2$. This defines a partial order on the set of all partial relation extensions of $C$, and we denote this poset by $\p$. Note that this poset has a unique minimal element $C$ and a unique maximal element $\Ctilde$.

We give another realization of the poset $\p$. Assume $B_1\le B_2$ in $\p$, where $B_1=C\ltimes E_1$ and $B_2=C\ltimes E_2$. There exists a $C$-$C$-bimodule $E_1'$ such that $E_2=E_1\oplus E_1'$. Using the same proof as in \cite[Lemma 2.1.1]{ABDLS}, we get that $B_2=B_1\ltimes E_1'$. This implies the existence of a Hochschild projection morphism $\varphi^1\colon\HH^1(B_2)\to\HH^1(B_1)$. We are now able to define the poset $\h$. Its elements are the first Hochschild cohomology groups $\HH^1(B)$ with $B$ a partial relation extension of $C$. We say that $\HH^1(B_1)$ is smaller than $\HH^1(B_2)$ whenever there exists a Hochschild projection morphism $\varphi^1\colon\HH^1(B_2)\to\HH^1(B_1)$.

\begin{cor}
 \label{cor poset}
 \begin{itemize}
\item [\textup{(a)}] The posets $\p$ and $\h$ are isomorphic.
\item [\textup{(b)}] The map $\textup{dim}\,\HH^1(-)\colon\p\to \mathbb{N}$ is a morphism of posets.
\end{itemize}
\end{cor}
 
\begin{proof}
 Statement (a) is clear from the respective definitions of our posets. In order to prove (b), we assume that 
 $B_1=C\ltimes E_1$ and $B_2=C\ltimes E_2$ are partial relation extensions, with $B_1$ smaller than $B_2$. We must prove that $\textup{dim}\,\HH^1(B_1)\le \textup{dim}\,\HH^1(B_2)$. Consider the diagram of Hochschild projection morphisms.
 \[\xymatrix{\HH^1(\Ctilde)\ar[rr]^{\varphi^1} \ar[rd]_{\varphi^1} && \HH^1(B_2)\ar[ld]^{\varphi^1}\\ 
&\HH^1(B_1) }
 \]
 Because the maps $\varphi^1$ are induced by the inclusions and projections, this diagram is commutative. It follows from our main theorem that the morphisms $\HH^1(\Ctilde)\to\HH^1(B_2)$ and $\HH^1(\Ctilde)\to\HH^1(B_1)$ are surjective. Therefore the morphism $\HH^1(B_2)\to\HH^1(B_1)$ is surjective and $\textup{dim}\,\HH^1(B_1)\le\textup{dim}\,\HH^1(B_2)$ as required.
 \end{proof}

\subsection{} We end the paper with a couple of examples. In both examples, $C$ is a tilted algebra, so that $\Ctilde $ is cluster-tilted.
\begin{example}
 \label{ex 1}
 Let $C$ be given by the quiver 
 
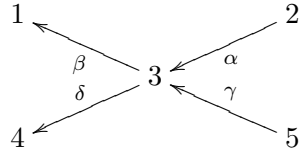
\[\xymatrix@R10pt@C40pt{1&&2\ar[ld]^\za\\ &3\ar[lu]^\zb\ar[ld]_\zd\\4&&5\ar[lu]_\zg}\] bound by the relations $\za\zb=0$ and $\zg\zd=0$. Then $\HH^1(C)=0$. 
 Let $B$ be the partial relation extension of $C$ given by the quiver
 
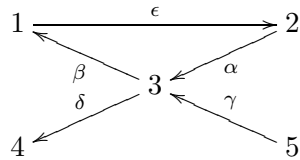
\[\xymatrix@R10pt@C40pt{1\ar[rr]^\ze&&2\ar[ld]^\za\\ &3\ar[lu]^\zb\ar[ld]_\zd\\4&&5\ar[lu]_\zg}\] 
 bound by the relations  $\za\zb=\zb\ze=\ze\za=0$ and  $\zg\zd=0$. Then $\HH^1(B)=k$.
 Finally, the relation extension $\Ctilde$ is given by the quiver
 
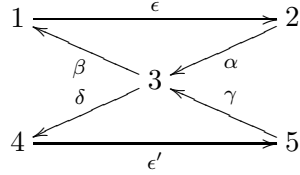
\[\xymatrix@R10pt@C40pt{1\ar[rr]^\ze&&2\ar[ld]^\za\\ &3\ar[lu]^\zb\ar[ld]_\zd\\4\ar[rr]_{\ze'}&&5\ar[lu]_\zg}\] 
 bound by the relations  $\za\zb=\zb\ze=\ze\za=0$ and  $\zg\zd=\zd\ze'=\ze'\zg=0$.  Clearly, $\HH^1(\Ctilde)= k^2$. \end{example}

\begin{example}
 \label{ex 2} Let $C$ be given by the quiver
 
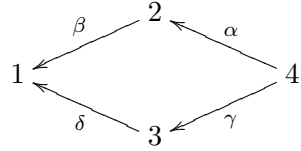
\[\xymatrix@R10pt@C40pt{&2\ar[ld]_\zb\\
 1&&4\ar[lu]_\za\ar[ld]^\zg\\&3\ar[lu]^\zd}\]
 bound by $\za\zb=0$ and $\zg\zd=0$.
 Let $B$ be the partial relation extension given by 
 
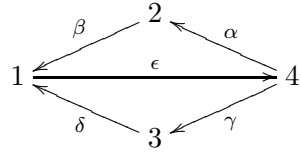
\[\xymatrix@R10pt@C40pt{&2\ar[ld]_\zb\\
 1\ar[rr]^\ze&&4\ar[lu]_\za\ar[ld]^\zg\\&3\ar[lu]^\zd}\]
 bound by $\za\zb=0$ and $\zg\zd=\zd\ze=\ze\zg=0$. 
 The relation extension $\Ctilde $ is given by
 
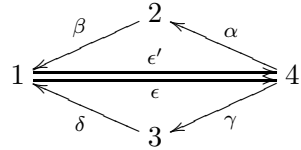
\[\xymatrix@R10pt@C40pt{&2\ar[ld]_\zb\\
 1\ar@<1.5pt>[rr]^{\ze'}\ar@<-1.5pt>[rr]_{\ze}
  &&4\ar[lu]_\za\ar[ld]^\zg\\&3\ar[lu]^\zd}\]
 bound by $\za\zb=\zb\ze'=\ze'\za=0$ and $\zg\zd=\zd\ze=\ze\zg=0$.
 Then $\HH^1(C)=k$ and $\HH^1(\Ctilde)=k^3$. In order to compute $\HH^1(B)$, observe first that ${}_CE'_{C}$ has a simple top (the arrow $\ze$) hence is indecomposable and so $\End_{C^e} E'=k$. We claim that $\hh^1(C,E')=0$. It suffices to prove that $\Der_0(C,E')=0$. Let $\xi$ be an arrow from $x$ to $y$ in $C$, and $d\colon C\to E'$ a derivation. Then $d(\xi)=d(e_x\xi e_y)=e_x d(\xi) e_y$, thus corresponds to a path from $x$ to $y$ in $B$ passing through $\ze$ and parallel to the arrow $\xi\in\{\za,\zb,\zg,\zd\}$. There is no such nonzero path. Therefore  $\hh^1(C,E')=0$ and so $\hh^1(B,E')=k$. It follows that $\HH^1(B)=k^2.$
 
 In this example we have $\cale(E'',B)\ne 0$. Indeed $E=E'\oplus E''=<\ze>\oplus <\ze'>$. Let $f\colon E''\to B$ be the morphism defined by $f(\ze')=\ze$. Then, for each $x\in E''$ we have $x\ze=0=\ze x$ because $x\ze, \ze x\in E^2=0.$ In particular $xf(\ze')+f(x)\ze'=0$ and thus $f$ is a nonzero element of $\cale(E'',B)$.
 \end{example}

\

\begin{thebibliography}{99}


%
%


\bibitem{ABS} I. Assem, T. Br\"{u}stle and R. Schiffler,  Cluster-tilted algebras as trivial extensions, \emph{Bull. Lond. Math. Soc.\/} {\bf 40} (2008), no. 1, 151--162. 

\bibitem{ABS2}
I. Assem, T. Br\"ustle, and R. Schiffler, Cluster-tilted algebras and slices, {\em J. of Algebra} {\bf 319} (2008) 3464--3479. 
 
\bibitem{ABDLS} I. Assem, J. C. Bustamante, J. Dionne, P. Le Meur and D. Smith,  	 Representation theory of partial relation extensions. arXiv:1604.01269
\bibitem{AGST} 
I. Assem, M. A. Gatica, R. Schiffler, and R. Taillefer, Hochschild cohomology of relation extension algebras, {\em J. Pure Appl. Alg.\/} {\bf 220}, 7, (2016) 2471--2499. 

\bibitem{AR} I. Assem and M.J. Redondo, The first Hochschild cohomology group of a Schurian cluster-tilted algebra, \textit{Manus. Math.} \textbf{128} (2009), no. 3, 373--388.

\bibitem{ARS} I. Assem,  M. J. Redondo and R. Schiffler, On the first Hochschild cohomology group of a cluster-tilted algebra, {\em Algebr. Represent. Theory\/} {\bf 18} (2015), no. 6, 1547--1576.

 \bibitem{AsScSe}  I. Assem, R. Schiffler and K. Serhiyenko, 
 Cluster-tilted and quasi-tilted algebras, {\it J. Pure Appl. Alg.\/} {\bf 221}, 9, (2017) 2266--2288.
  \bibitem{AsScSe2}  I. Assem, R. Schiffler and K. Serhiyenko, 
Modules over cluster-tilted algebras that do not lie on local slices,  {\em Archiv Math.\/} (2017). doi.org/10.1007/s00013-017-1121-5
\bibitem{ASS} I. Assem, D. Simson  and A. Skowro\'nski, Elements of the Representation Theory of Associative Algebras, \textit{London Math. Soc. Student Texts} \textbf{65} (2006), Cambridge University Press.


\bibitem{BMR}  A.B. Buan, R. Marsh and I. Reiten,  Cluster-tilted algebras, \emph{Trans. Amer. Math. Soc. \/} {\bf 239} (2007), no. 1, 323--332. 

\bibitem{BMR2}  { A. B. Buan, R. Marsh and I. Reiten},
 Cluster-tilted algebras of finite representation type, \emph{
 J. Algebra}  {\bf 306}  (2006),  no. 2, 412--431.  
 

\bibitem{CCS} P. Caldero, F. Chapoton and R. Schiffler,  Quivers with relations arising from clusters ($A_n$ case), \emph{Trans. Amer. Math. Soc. \/} {\bf 358} (2006), no. 3, 1347--1363. 
%
%
%
%
%
%
%
%
%
%
%
%


\bibitem{Happel} D. Happel, Hochschild cohomology of finite dimensional algebras,   \emph{Springer Lecture Notes in Math. \/} {\bf 1404} (1989), 108--126.

%
%

%
%
%
\bibitem{S}  R. Schiffler, Quiver Representations, \emph{CMS Books in Mathematics}, Springer Verlag, (2014), ISBN 978-3-319-09203-4. 
\bibitem{SS}  R. Schiffler and K. Serhiyenko, Induced and coinduced modules over cluster-tilted algebras,  {\em J. Algebra\/ }{\bf 472} (2017) 226--258. 
 \bibitem{SS2}  R. Schiffler and K. Serhiyenko, Injective presentations of induced  modules over cluster-tilted algebras, {\em Algebr. Represent. Theory\/} (2017). doi.org/10.1007/s10468-017-9721-0
 \end{thebibliography}
\end{document}